\documentclass[a4paper,12pt]{article}

\usepackage[intlimits]{amsmath}
\usepackage{amssymb}
\usepackage{amsthm}
\usepackage{mathrsfs}
\usepackage{cite}
\usepackage{mathrsfs}
\usepackage{color}
\usepackage[normalem]{ulem}

\newcommand{\la}{\langle}
\newcommand{\ra}{\rangle}
\newcommand{\pr}{\partial}
\newcommand{\dom}{\Omega}

\newcommand{\re}{\mathfrak{Re}\;}

\newcommand{\tr}{{\mathfrak{Tr}\,}}

\newcommand{\X}{{\mathcal{X}}}
\newcommand{\s}{{\mathcal{S}}}
\newcommand{\T}{{\mathcal{T}}}

\newcommand{\R}{\mathbb{R}}
\newcommand{\C}{\mathbb{C}}

\newtheorem{thm}{Theorem}
\newtheorem{prop}{Proposition}[section]
\newtheorem{lem}{Lemma}[section]
\newtheorem{cor}{Corollary}[section]

\title{Recovery of coefficients for a weighted p-Laplacian perturbed by a linear second order term}

\author{C\u{a}t\u{a}lin I. C\^{a}rstea\thanks{School of Mathematics, Sichuan University, Chengdu, Sichuan, 610064, P.R.China; email: catalin.carstea@gmail.com} \and Manas Kar\thanks{Indian Institute of Science Education and Research Bhopal, Indore By-pass Road, Bhauri Bhopal, 462066,
Madhya Pradesh, India; email: manas@iiserb.ac.in}}\date{}

\begin{document}
\maketitle

\begin{abstract}
This paper considers the inverse boundary value problem for the equation $\nabla\cdot(\sigma\nabla u+a|\nabla u|^{p-2}\nabla u)=0$. We give a procedure for the recovery of the coefficients $\sigma$ and $a$ from the corresponding Dirichlet-to-Neumann map, under suitable regularity and ellipticity assumptions.
\end{abstract}

\section{Introduction}

Let $\dom\subset\R^n$, $n\geq3$, be a bounded domain with smooth boundary $\pr\dom$. In this domain, we may consider the following quasilinear boundary value problem for the function $u$
\begin{equation}\label{eq-general}
\left\{\begin{array}{l} \nabla\cdot\left(J(x,u,\nabla u)\right)=0,\\[5pt] u|_{\pr\dom}=f.\end{array}\right.
\end{equation}
Here we assume $J$ is a vector valued function.

In the linear case $J(x,u,\nabla u)=\sigma(x) \nabla u$, with a function $\sigma$ that has both upper and lower bounds that are positive, \eqref{eq-general} may be interpreted as the equation describing conduction in an object whose spatial extension coincides with $\dom$, and whose (possibly inhomogeneous) conductivity matrix is $\sigma$. In this case $u$ would be the electric potential (with boundary values $f$) and $J(x,\nabla u(x))=\sigma(x)\nabla u(x) $ would be the current density, as given by Ohm's law. The general case could then be seen (for example, and not exclusively) as describing nonlinear conduction phenomena.

Assuming the equation \eqref{eq-general} can be solved for some class of Dirichlet boundary data, with solutions having suitable regularity, we can define the Dirichlet-to-Neumann map (which may also be called the "potential-to-current map" in the conduction  interpretation of the equation) to be
\begin{equation}
\left.\Lambda_{J} f=\nu\cdot J\left(x,u,\nabla u\right)\right|_{\pr\dom},
\end{equation}
where $u$ is the solution to equation \eqref{eq-general} and $\nu$ is the unit outer pointing normal vector to the boundary $\pr\dom$. From an experimental point of view, $\Lambda_J$ represents all the information that may be obtained from boundary measurements of pairs of Dirichlet and Neumann data (potential and current, if the equation is describing conduction).

The inverse problem proposed by Calder\'on in \cite{Ca} is to invert the mapping $J\to\Lambda_J$. An important subproblem is the question of uniqueness, i.e. the question of the injectivity of the mapping $J\to\Lambda_J$. In the linear case, much work has already been done towards solving these problems, especially in the isotropic case (i.e. when $\sigma$ is a multiple of the identity matrix). For $n\geq 3$, uniqueness was proven in \cite{SU} and a method for reconstructing $\sigma$ from the Dirichlet-to-Neumann map was given in \cite{N}. The anisotropic case is more difficult, as uniqueness does not hold, and the problem is still open when $n\geq 3$. 

In the case of  elliptic semilinear or quasilinear equations, there are a number of uniqueness results that are known. For the semilinear case, see \cite{IN}, \cite{IS}, \cite{S2}, \cite{FO}, \cite{LLLS1}, \cite{LLLS2}, \cite{KU}. For the quasilinear case, not in divergence form, see \cite{I2}. For the quasilinear case in divergence form, when $J(x, u, \nabla u)= c(x,u)\nabla u$ see \cite{S1}, \cite{SuU}, \cite{EPS}; when $J(x,u,\nabla u)= A(u,\nabla u)$, see \cite{MU}, \cite{Sh}; when $J(x,u,\nabla u)=A(x,\nabla u)$ in 2D see \cite{HS}; when $J(x,u,\nabla u)=\sigma(x)\nabla u+b(x)|\nabla u|^2$, see \cite{KN}, \cite{CNV}. In this list, the functions $c$, $\sigma$ are meant to be scalar valued and the functions $A$, $b$ are meant to be vector valued. Note that \cite{CNV} also gives a reconstruction method for the coefficients $\sigma$ and $b$. We can also mention the uniqueness results of \cite{AZ}, \cite{C} for quasilinear time-harmonic Maxwell systems.

In this paper we will consider the following boundary value problem for the complex valued function $u$
\begin{equation}\label{eq}
\left\{\begin{array}{l}\nabla\cdot\left(\sigma(x)\nabla u+a(x)|\nabla u|^{p-2}\nabla u\right)=0,\\[5pt]
u|_{\pr\dom}=f,\end{array}\right.
\end{equation}
where $p\in (1,\infty)\setminus\{2\}$,
\begin{equation}
|\nabla u|^2=\nabla u\cdot \overline{\nabla u},
\end{equation}  
\begin{equation}\label{assumptions-1}
\sigma\in C^\infty(\bar\dom),\quad a \in L^\infty(\dom),
\end{equation}
and for some positive constants $\lambda, m>0$ we have
\begin{equation}\label{assumptions-2}
0<\lambda<\sigma<\lambda^{-1},\quad 0<m<a<m^{-1}.
\end{equation}
 This can be interpreted as a nonlinear conductivity equation, with nonlinear conductivity
\begin{equation}
\gamma(x,\nabla u)= \sigma(x)+a(x)|\nabla u|^{p-2}.
\end{equation}
The associated Dirichlet-to-Neumann map is
\begin{equation}
\left.\Lambda_{\sigma,a,p} f=\left(\sigma(x)+a(x)|\nabla u|^{p-2}\right)\pr_\nu u\right|_{\pr\dom},
\end{equation}
where $u$ is the solution to \eqref{eq}. From here on, $q$ will be the H\"older conjugate of $p$, i.e. $q^{-1}+p^{-1}=1$.

The nonlinear term in equation \eqref{eq} is known as the weighted p-Laplacian. Inverse problems for the $\sigma=0$ case have been investigated in \cite{BKS}, \cite{B}, \cite{BHKS}, \cite{GKS}, \cite{BIK}, \cite{KW}. Except for the two dimensional case, a uniqueness result has not been obtained yet. Part of the interest of this paper lies in showing that the addition of a linear term to the equation renders the problem much more tractable.

We also need to mention the paper \cite{HHM}, where the authors consider the inverse problem for the equation
\begin{equation}
\nabla(\sigma(\tau^2+|\nabla u|^2)^\frac{p-2}{2}\nabla u)=0,
\end{equation}
from a numerics point of view. We can isolate the linear part of their equation by writing it as
\begin{equation}
\nabla(\sigma\nabla u)+
\nabla\left[ \sigma\left((1+\tau^{-2}|\nabla u|^2)^{\frac{p-2}{2}} -1\right)\nabla u \right] =0,
\end{equation}
where the second term is clearly purely nonlinear. Though not identical to the equation we are considering here, this model can be treated in a similar way to what we will present in this paper.

\subsection{Main results and outline}

In order to discuss the inverse problem, we must first give an account of when and in what sense the equation \eqref{eq} can be solved.  Using energy minimization methods it is relatively straightforward to prove the existence of weak solutions. To this end we introduce the spaces of Dirichlet data 
\begin{equation}
\X_p=\left\{\begin{array}{l} W^{1-\frac{1}{p},p}(\pr\dom)\text{ if }p>2,\\[5pt]
W^{\frac{1}{2},2}(\pr\dom)\text{ if }p\in(1,2). \end{array}\right.
\end{equation}
We will denote by $\X_p'$ the dual of the space $\X_p$. We will show in section \ref{forward} that the Dirichlet-to-Neumann map $\Lambda_{\sigma,a,p}$ maps $\X_p$ into $\X_p'$. 

For $p>2$ and Dirichlet boundary data  $f\in W^{2-\frac{1}{p},p}(\pr\dom)$ with sufficiently small norm, we will show the existence of strong $W^{2,p}(\dom)$ solutions of \eqref{eq}. This will be done via a contraction principle argument. In the $p\in(1,2)$ case, we do not need to use strong solutions in our treatment of the inverse problem, so we forgo investigating this topic here.

Our main result will be

\begin{thm}\label{thm1}
If the assumptions \eqref{assumptions-1}, \eqref{assumptions-2} hold, then the parameters $\sigma$, $a$, and $p$ can be reconstructed from the Dirichlet-to-Neumann map $\Lambda_{\sigma,a,p}$.
\end{thm}

As in many other works on inverse boundary value problems for semilinear and quasilinear elliptic equations, we will use a version of a so called ``second linearization'' trick originally used in \cite{I1}. In our case, when $p>2$, this amounts to considering solutions $u_\epsilon$ of \eqref{eq} with Dirichlet data $\epsilon f$, where $\epsilon$ is a small parameter. One can show that these solutions have an asymptotic expansion of the form
\begin{equation}
u_\epsilon=\epsilon u_0+\epsilon^{p-1}u_1+\mathscr{O}(\epsilon^{2p-3}),
\end{equation}
as $\epsilon\to0$.
Then the Dirichlet-to-Neumann map will also have a similar expansion. For $w$ a solution to $\nabla\cdot(\sigma\nabla w)=0$ we have
\begin{equation}
\la\Lambda_{\sigma,a,p}(\epsilon f),w|_{\pr\dom}\ra=\epsilon\la\Lambda_\sigma(f),w\ra+\epsilon^{p-1}I(u_0,\overline{w})+\mathscr{O}(\epsilon^{2p-3}),
\end{equation}
where $\Lambda_\sigma$ is the Dirichlet-to-Neumann map for the linear problem, $u_0$ is the solution of the linear equation with Dirichlet data $f$, and $I(u_0,\overline{w})$ is a map that is homogeneous of order $p-1$ in $u_0$. Both $\Lambda_\sigma$ and $I$ are determined by the full Dirichlet-to-Neumann map $\Lambda_{\sigma,a,p}$. The coefficient $\sigma$ can then be reconstructed from $\Lambda_\sigma$ using the method of \cite{N}. 

It is clear that $p$ is also determined by $\Lambda_{\sigma,a,p}$ since it appears in the exponent of $\epsilon$ in the second term of the expansion.

In past work for semilinear or quasilinear equations with polynomial type nonlinearities (e.g. \cite{KN}, \cite{AZ}, \cite{C}, \cite{CNV}, \cite{FO}, \cite{LLLS1}, \cite{LLLS2}) a polarization trick has been used on the equivalent of our term $I$, in order to obtain an integral identity involving multiple independently chosen solutions of the linear equation. Since here $p$ is not necessarily an integer, we use a trick inspired by polarization to obtain our final integral identity. We then use a method developed in \cite{CNV} to obtain a reconstruction of the coefficient $a$. All this is covered in section \ref{reconstruction-1}.

When $p\in(1,2)$, the above method needs to be modified since in the limit $\epsilon\to0$ the nonlinear term (which is now sublinear) would be the more significant one, so we would not be able to derive asymptotic expansions in the same way. Instead, we take Dirichlet data of the form $\epsilon^{-1}f$ and eventually obtain the expansion
\begin{equation}
\la\Lambda_{\sigma,a,p}(\epsilon^{-1} f),w|_{\pr\dom}\ra
=\epsilon^{-1}\la\Lambda_\sigma,w|_{\pr\dom}\ra+\epsilon^{1-p}I(u_0,\overline{w})+\mathscr{O}(\epsilon^{3-2p}).
\end{equation}
From here the method proceeds identically. We cover this case in section \ref{reconstruction-2}

Some vector identities that we need to use throughout the paper are collected in appendix \ref{appendix-A}.

\section{The forward problem}\label{forward}

In this section we will derive existence results and estimates for solutions of equation \eqref{eq} with Dirichlet data in $\X_p$. Since we are principally concerned with the inverse problem, we do not attempt to obtain a comprehensive existence and regularity theory for this equation. Instead, we only derive results that are sufficient for our subsequent arguments. In particular, we will prove the existence of weak solutions for all $p\in(1,\infty)$. However, we will only derive the existence of strong solutions in the $p>2$ case, and only for sufficiently small source and boundary data.

\subsection{Weak solutions}

\begin{prop}\label{weak-existence}
If $p\in(1,\infty)$, $f\in {\X_p}  $, and $F\in L^q(\dom)$, then the boundary value problem 
\begin{equation}\label{eq-F}
\left\{\begin{array}{l}\nabla\cdot\left(\sigma(x)\nabla u+ a(x)|\nabla u|^{p-2}\nabla u\right)=\nabla F,\\[5pt]
u|_{\pr\dom}=f,\end{array}\right.
\end{equation}
  has a unique solution $u\in W^{1,2}(\dom)\cap W^{1,p}(\dom)$. Furthermore, there exists $C>0$, independent of $f$, $F$ such that
\begin{equation}
||u||_{W^{1,p}(\dom)}\leq C\left(||f||_{\X_p}+||F||_{L^q(\dom)}^{\frac{1}{p-1}}\right),\text{ if }p>2,
\end{equation} 
\begin{equation}
||u||_{W^{1,2}(\dom)}\leq C\left(||f||_{\X_p}+||F||_{L^q(\dom)}\right),\text{ if }p\in(1,2).
\end{equation} 
\end{prop}

\begin{proof}

Let $\s=W^{1,2}(\dom)\cap W^{1,p}(\dom)$ be equipped with the norm
\begin{equation}
||u||_\s = \max (||u||_{W^{1,2}(\dom)},||u||_{W^{1,p}(\dom)}).
\end{equation}
With this norm, $\s$ is a Banach space. Its dual is $\s'=W^{1,p}(\dom)'+W^{1,2}(\dom)'$, with the norm
\begin{multline}
||l||_{\s'}
=\inf\{||l_1||_{(W^{1,2}(\dom))'}+||l_2||_{(W^{1,p}(\dom))'}:l=l_1+l_2,\\[5pt]
l_1\in W^{1,2}(\dom)', l_2\in W^{1,p}(\dom)'\}.
\end{multline}
The space $\s$ is reflexive. (For example, see \cite{BS}.)

Define the energy
\begin{equation}
E[v]=\int_\dom \frac{1}{2}\sigma(x)|\nabla v|^2+\frac{1}{p}a(x)|\nabla v|^p+\re\nabla\bar v\cdot F, \quad v\in \s.
\end{equation}
Note that
\begin{multline}
E[v]\geq\int_\dom\left(\frac{\lambda}{2}|\nabla v|^2+\frac{m }{2p}|\nabla v|^p\right)-C||F||_{L^q(\dom)}^q\\[5pt]
\geq C_1\left[\max(||\nabla v||_{L^2(\dom)}^2,||\nabla v||_{L^p(\dom)}^p)-||F||_{L^q(\dom)}^q   \right],
\end{multline}
and, similarly
\begin{equation}
E[v]\leq C_2 \left[\max(||\nabla v||_{L^2(\dom)}^2,||\nabla v||_{L^p(\dom)}^p)+||F||_{L^q(\dom)}^q   \right].
\end{equation}
Here $C_1, C_2>0$ are constants that depend on $\lambda$, $m$, $p$,  and $|\dom|$.

Let $u_k\in\s$ be a minimizing sequence such that $u_k|_{\pr\dom}=f$ and
\begin{equation}
\lim_{k\to\infty} E[u_k]=\min_{v\in \s,v|_{\pr\dom}=f} E[v].
\end{equation}
Since $u_k$ is bounded in $\s$, up to extracting a subsequence, it converges weakly to a limit $u\in\s$. By the convexity in $\nabla v$ of the integrand in the definition of $E[v]$, we have that $E$ is weakly lower semicontinuous (see \cite[Section 8.2]{E}). It is then clear that $E[u]=\min_{v\in\s ,v|_{\pr\dom}=f} E[v]$ and (by the  convexity of $E$) it is the unique minimum. 

Let $\phi\in C_0^\infty(\dom)$ and, for $z\in\C$ we should have
\begin{equation}
\pr_{\bar z} E[u+z\phi]|_{z=0}=\pr_{z} E[u+z\phi]|_{z=0}=0.
\end{equation} 
In particular, since
\begin{equation}
\pr_{\bar z}|\nabla(u+z\phi)|^2|_{z=0}=\nabla\bar\phi\cdot\nabla u,
\end{equation}
\begin{equation}
\pr_{\bar z}|\nabla (u+z\phi)|^p|_{z=0}=\pr_{\bar z} \left(|\nabla(u+z\phi)|^2\right)^{p/2}|_{z=0}=\frac{p}{2}|\nabla u|^{p-2}\nabla u\cdot\nabla\bar\phi,
\end{equation}
it follows that
\begin{equation}
\int_\dom \nabla \bar\phi\cdot\left(\sigma(x)\nabla u+  a(x)|\nabla u|^{p-2}\nabla u+F\right)=0.
\end{equation}
Therefore, $u$ is a weak solution to equation \eqref{eq-F}.

Conversely, if $u$ is a weak solution to equation \eqref{eq-F}, then for any $w\in W^{1,2}_0(\dom)\cap W_0^{1,p}(\dom)$ we have that
\begin{equation}
\int_\dom \nabla \bar w\cdot \left(\sigma(x)\nabla u+  a(x)|\nabla u|^{p-2}\nabla u+F\right)=0.
\end{equation}
This, and its complex conjugate, give that
\begin{equation}
\pr_{\bar z} E[u+zw]|_{z=0}=\pr_{z} E[u+zw]|_{z=0}=0,
\end{equation} 
so, by the uniqueness of the energy minimum, we conclude the uniqueness of the solutions to \eqref{eq}.

Suppose $p>2$. We define the functionals
\begin{equation}
E_2[v]=\int_\dom\frac{1}{2}\sigma(x)|\nabla v|^2,\text{ and } E_p[v]=\int_\dom\frac{1}{p}a(x)|\nabla v|^p+\re\nabla\bar v\cdot F.
\end{equation}
Consider now the function $v_2\in W^{1,p}(\dom)$ which is the solution of the boundary value problem
\begin{equation}
\left\{\begin{array}{l}\nabla\left(\sigma(x)\nabla v_2\right)=0,\\[5pt]
v_2|_{\pr\dom}=f.\end{array}\right.
\end{equation}
We  know that
\begin{equation}
||v_2||_{W^{1,p}(\dom)}\leq C||f||_{\X_p}
\end{equation}
and that it minimizes the energy $E_2$.
 Indeed, regarding the last claim, suppose $w\in W^{1,p}_0(\dom)$. Then
\begin{equation}
E_2[v_2+w]=E[v_2]+\re\int_\dom \nabla\bar w\cdot\left(\sigma\nabla v_2\right)+\int_\dom\frac{1}{2}\sigma|\nabla w|^2\geq E[v_2],
\end{equation}
as the middle term is zero and the last term is positive.

Since $E[v]=E_2[v]+E_p[v]$, it is clear that
\begin{equation}
E[u]\leq E_2[v_2]+E_p[v_2].
\end{equation}
Then
\begin{equation}
E_p[v_2]\geq E[u]-E_2[v_2]=E_2[u]-E_2[v_2]+E_p[u]\geq E_p[u],
\end{equation}
so 
\begin{equation}
E_p[u]\leq E_{p}[v_2].
\end{equation}
From this it follows that
\begin{multline}
||\nabla u||_{L^p(\dom)}^p\leq C\left(||\nabla v_2||_{L^p(\dom)}^p+ \left|\int_\dom\nabla\bar v\cdot F\right|+ \left|\int_\dom\nabla\bar v_2\cdot F\right|\right)\\[5pt]
\leq C\left(||\nabla v_2||_{L^p(\dom)}^p+ ||F||_{L^q(\dom)}^q\right)+\frac{1}{2}||\nabla u||_{L^p(\dom)}^p,
\end{multline}
so
\begin{equation}
||\nabla u||_{L^p(\dom)}\leq C\left(||f||_{\X_p}+ ||F||_{L^q(\dom)}^{\frac{1}{p-1}}\right).
\end{equation}

Suppose now that $p\in(1,2)$. Let
\begin{equation}
\tilde E_2[v]=\int_\dom\frac{1}{2}\sigma(x)|\nabla v|^2+\re\nabla\bar v\cdot F,\text{ and } \tilde E_p[v]=\int_\dom\frac{1}{p}a(x)|\nabla v|^p.
\end{equation}
Suppose 
\begin{equation}
\left\{\begin{array}{l}\nabla\left(\sigma(x)\nabla \tilde v_2\right)=\nabla F,\\[5pt]
\tilde v_2|_{\pr\dom}=f.\end{array}\right.
\end{equation}
Then
\begin{equation}
||\tilde v_2||_{W^{1,2}(\dom)}\leq C(||f||_{\X_p}+||F||_{L^q(\dom)}),
\end{equation}
and $\tilde v_2$ minimizes the energy $\tilde E_2$.

It still holds that $\tilde E_p[u]\leq \tilde E_p[\tilde v_2]$, so 
\begin{equation}
||\nabla u||_{L^p(\dom)}\leq C||\nabla v_2||_{L^p(\dom)}\leq C||\nabla v_2||_{L^2(\dom)},
\end{equation}
and the same conclusion follows.
\end{proof}

\begin{cor}
The DN map $\Lambda_{\sigma,a,p}:{\X_p}  \to {\X_p}  '$ and
\begin{equation}
||\Lambda_{\sigma,a,p}f||_{{\X_p}  '}\leq C \left(||f||_{\X_p}+||f||_{\X_p}^{p-1}\right).
\end{equation}
\end{cor}
\begin{proof}
{\bf case i)} $p>2$. Let $w\in W^{1,p}(\dom)$. Then
\begin{multline}
\left|\la\Lambda_{\sigma,a,p}f,w|_{\pr\dom}\ra\right|=\left|\int_\dom(\sigma(x)+a(x)|\nabla u|^{p-2})\nabla u\cdot\nabla\bar w\right|\\[5pt]
\leq\lambda^{-1}||\nabla u||_{L^2(\dom)}||\nabla w||_{L^2(\dom)}+m^{-1}||\nabla u||_{L^p(\dom)}^{p-1}||\nabla w||_{L^p(\dom)}\\[5pt]
\leq C (||u||_{W^{1,p}(\dom)}+||u||_{W^{1,p}(\dom)}^{p-1})||w||_{W^{1,p}(\dom)}.
\end{multline}
{\bf case ii)} $p\in (1,2)$. Let $w\in W^{1,2}(\dom)$. Then, similarly to the above estimate,
\begin{equation}
\left|\la\Lambda_{\sigma,a,p}f,w|_{\pr\dom}\ra\right|
\leq C(||u||_{W^{1,2}(\dom)}+||u||_{W^{1,2}(\dom)}^{p-1})||w||_{W^{1,2}(\dom)}.
\end{equation}
\end{proof}

\subsection{Strong solutions}

Next we will consider strong solutions to our equation, in the $p>2$ case.
\begin{thm}\label{thm-strong}
Suppose $r>n$, $p>2$. Let $F\in L^r(\dom)$ and $f\in W^{2-\frac{1}{r},r}(\pr\dom)$. There exists $\delta>0$ so that if 
\begin{equation}
||F||_{L^r(\dom)}, ||f||_{W^{2-\frac{1}{r},r}(\pr\dom)}<\delta,
\end{equation}
then equation 
\begin{equation}\label{eq-F-s}
\left\{\begin{array}{l}\nabla\cdot\left(\sigma(x)\nabla u+ a(x)|\nabla u|^{p-2}\nabla u\right)=\nabla F,\\[5pt]
u|_{\pr\dom}=f,\end{array}\right.
\end{equation}
 has a unique solution $u\in W^{2,r}(\dom)$. Furthermore, there exists $C>0$ such that
\begin{equation}
||u||_{W^{2,r}(\dom)}\leq C\left(||f||_{W^{2-\frac{1}{r},r}(\pr\dom)}+||F||_{L^r(\dom)}\right).
\end{equation}
\end{thm}
\begin{proof}
Let $G_\sigma$ be the inverse of the operator $\nabla\cdot(\sigma\nabla\cdot)$, with zero Dirichlet boundary conditions. That is, if
\begin{equation}
\left\{\begin{array}{l}\nabla\cdot(\sigma\nabla v)=S,\text{ in }\dom,\\ v|_{\pr\dom}=0,\end{array}\right.
\end{equation}
then $G_\sigma(S)=v$. $G_\sigma$ is a bounded operator form $L^r(\dom)$ into $W^{2,r}_0(\dom)$.

Suppose $f\in W^{2-\frac{1}{r},r}(\pr\dom)$ and let $v_f\in W^{2,r}(\dom)$ be the solution of 
\begin{equation}
\left\{\begin{array}{l}\nabla\cdot(\sigma\nabla v_f)=F,\text{ in }\dom,\\ v_f|_{\pr\dom}=f.\end{array}\right.
\end{equation}
There exists a constant $C>0$ such that 
\begin{equation}
||v_f||_{W^{2,r}(\dom)}\leq C\left(||f||_{W^{2-\frac{1}{r},r}(\dom)}+||F||_{L^r(\dom)}\right).
\end{equation}
For $v\in W^{2,r}_0(\dom)$ we define the operator
\begin{equation}
\T(v)=-G_\sigma\left[\nabla\cdot\left(a|\nabla(v_f+v)|^{p-2}\nabla(v_f+v)  \right)\right].
\end{equation}

Let $A(v)$, $B(v)$, and $H(v)$ be the matrix valued functions with coefficients
\begin{equation}
A_{jk}(v)=|\nabla(v_f+v)|^{p-4}\pr_j(v_f+v)\pr_k(\overline{v_f}+\overline{v}),
\end{equation}
\begin{equation}
B_{jk}(v)=|\nabla(v_f+v)|^{p-4}\pr_j(v_f+v)\pr_k({v_f}+{v}),
\end{equation}
\begin{equation}
H_{jk}(v)=\pr_{jk}(v_f+v).
\end{equation}
Then
\begin{multline}
\nabla\cdot\left(a|\nabla(v_f+v)|^{p-2}\nabla(v_f+v)  \right)=
|\nabla(v_f+v)|^{p-2}(\nabla a)\cdot\nabla(v_f+v)\\[5pt]+
a|\nabla(v_f+v)|^{p-2}\triangle(v_f+v)
+\frac{p-2}{2}a\tr\left(AH+B\overline{H} \right),
\end{multline}
which belongs to $L^r(\dom)$, since, by  Sobolev embedding, $W^{1,r}(\dom)\subset L^\infty(\dom)$.
We then see that $\T:W^{2,r}_0(\dom)\to W^{2,r}_0(\dom)$. 

Let $v_1, v_2\in W^{2,r}_0(\dom)$. Then, applying Lemma \ref{vector-estimates}, we have
\begin{multline}
\left\Vert\,|\nabla(v_f+v_1)|^{p-2}\nabla(v_f+v_1)-|\nabla(v_f+v_2)|^{p-2}\nabla(v_f+v_2)\right\Vert_{L^r(\dom)}\\[5pt]
\leq C\left(||\nabla v_f||_{L^\infty(\dom)}+||\nabla v_1||_{L^\infty(\dom)}+||\nabla v_2||_{L^\infty(\dom)}  \right)^{p-2}||\nabla(v_1-v_2)||_{L^r(\dom)}\\[5pt]
\leq C\left(||f||_{W^{2-\frac{1}{r},r}(\pr\dom)}+||F||_{L^r(\dom)}+||v_1||_{W^{2,r}_0(\dom)}+||v_2||_{W^{2,r}_0(\dom)}  \right)^{p-2}||v_1-v_2||_{W^{2,r}_0(\dom)}.
\end{multline}
Since
\begin{multline}
|\nabla(v_f+v_1)|^{p-2}\triangle(v_f+v_1)-|\nabla(v_f+v_2)|^{p-2}\triangle(v_f+v_2)\\[5pt]
= \left( |\nabla(v_f+v_1)|^{p-2}-|\nabla(v_f+v_2)|^{p-2}\right)\triangle(v_f+v_1)\\[5pt]
+|\nabla(v_f+v_2)|^{p-2}\triangle(v_1-v_2),
\end{multline}
and, by Lemma \ref{p-2},
\begin{multline}
||\nabla(v_f+v_1)|^{p-2}-|\nabla(v_f+v_2)|^{p-2}||_{L^\infty(\dom)}\\[5pt]
\leq C \left(||f||_{W^{2-\frac{1}{r},r}(\pr\dom)}+||F||_{L^r(\dom)}+||v_1||_{W^{2,r}_0(\dom)}+||v_2||_{W^{2,r}_0(\dom)}  \right)^{\mu_1(p-2)}\\[5pt]\times ||v_1-v_2||_{W^{1,r}(\dom)}^{\min(p-2,1)},
\end{multline}
we also have 
\begin{multline}
\left\Vert|\nabla(v_f+v_1)|^{p-2}\triangle(v_f+v_1)-|\nabla(v_f+v_1)|^{p-2}\triangle(v_f+v_1)\right\Vert_{L^r(\dom)}\\[5pt]
\leq C\left(||f||_{W^{2-\frac{1}{r},r}(\pr\dom)}+||F||_{L^r(\dom)}+||v_1||_{W^{2,r}_0(\dom)}+||v_2||_{W^{2,r}_0(\dom)}  \right)^{p-2}||v_1-v_2||_{W^{2,r}_0(\dom)}\\[5pt]
+C \left(||f||_{W^{2-\frac{1}{r},r}(\pr\dom)}+||F||_{L^r(\dom)}+||v_1||_{W^{2,r}_0(\dom)}+||v_2||_{W^{2,r}_0(\dom)}  \right)^{\mu_1(p-2)}\\[5pt]\times ||v_1-v_2||_{W^{1,r}(\dom)}^{\min(p-2,1)}.
\end{multline}
Finally, applying Lemma \ref{holder}, we have
\begin{multline}
\left\Vert \tr\left(A(v_1)H(v_1)+B(v_1)\overline{H}(v_1)-A(v_2)H(v_2)-B(v_2)\overline{H}(v_2)\right)\right\Vert_{L^r(\dom)}\\[5pt]
\leq C\left(||f||_{W^{2-\frac{1}{r},r}(\pr\dom)}+||F||_{L^r(\dom)}+||v_1||_{W^{2,r}_0(\dom)} \right)^{p-2}||v_1-v_2||_{W^{2,r}_0(\dom)}\\[5pt]
+C\left(||f||_{W^{2-\frac{1}{r},r}(\pr\dom)}+||F||_{L^r(\dom)}+||v_2||_{W^{2,r}_0(\dom)} \right)^{p-2}\\[5pt]\times
(||f||_{W^{2-\frac{1}{r},r}(\pr\dom)}+||F||_{L^r(\dom)}+||v_1||_{W^{2,r}_0(\dom)}+||v_2||_{W^{2,r}_0(\dom)} )^{\mu_2(p-2)} ||v_1-v_2||_{W^{2,r}_0(\dom)}^{s},
\end{multline}
where $s=\frac{1}{4}\min({p-2},1)$.

We see now that there is a $\delta>0$ such that if 
\begin{equation}
||F||_{L^r(\dom)}, ||f||_{W^{2-\frac{1}{r},r}(\pr\dom)}<\delta,
\end{equation}
then $\T$ is a contraction (taking the distance to be $d(v_1,v_2)=||v_1-v_2||_{W^{2,r}_0(\dom)}^{s}$) on the ball of radius $\delta$ in $W^{2,r}_0(\dom)$. That is, there exists $\kappa\in(0,1)$ such that if $||v_1||_{W^{2,r}_0(\dom)}, ||v_2||_{W^{2,r}_0(\dom)}<\delta$, then
\begin{equation}
||\T(v_1)-\T(v_2)||_{W^{2,r}_0(\dom)}^s\leq\kappa ||v_1-v_2||_{W^{2,r}_0(\dom)}^s
 \end{equation} 
 Let $v\in W^{2,r}_0(\dom)$ be the fixed point of this contraction. Then $u=v_f+v$ is the solution whose existence we needed to prove.

Furthermore,
\begin{multline}
||v||_{W^{2,r}_0(\dom)}\leq ||\T(v)-\T(0)||_{W^{2,r}_0(\dom)}+||\T(0)||_{W^{2,r}_0(\dom)}\\[5pt]
\leq \kappa^{\frac{1}{s}}||v||_{W^{2,r}_0(\dom)}+C\left(||f||_{W^{2-\frac{1}{r},r}(\pr\dom)}+||F||_{L^r(\dom)}\right)^{p-1},
\end{multline}
which implies that
\begin{equation}
||v||_{W^{2,r}_0(\dom)}\leq C\left(||f||_{W^{2-\frac{1}{r},r}(\pr\dom)}+||F||_{L^r(\dom)}\right)^{p-1}.
\end{equation}
This gives
\begin{equation}
||u||_{W^{2,r}(\dom)}\leq C\left(||f||_{W^{2-\frac{1}{r},r}(\pr\dom)}+||F||_{L^r(\dom)}+ \left(||f||_{W^{2-\frac{1}{r},r}(\pr\dom)}+||F||_{L^r(\dom)}\right)^{p-1}\right).
\end{equation}
\end{proof}

\section{Reconstruction when $p>2$}\label{reconstruction-1}

In this section we will give a reconstruction procedure for the parameters $\sigma$, $a$, and $p$, in the case when $p>2$. We do this by imposing boundary data multiplied by a parameter $\epsilon$ and then using the asymptotic expansion of the Dirichlet-to-Neumann map in $\epsilon$, as $\epsilon\to0$. Using he order $\epsilon^{2p-3}$ part of the DN map we obtain an integral identity involving $a$ and several solutions of the linear part of equation \eqref{eq}. Pluging complex geometric optics  solutions, constructed in \cite{SU}, into this identity allows us to reconstruct the coefficient $a$.

\subsection{Asymptotics of solutions}

For  $0<\epsilon<1$, and $f\in W^{2-\frac{1}{r},r}(\pr\dom)$ with $||f||_{W^{2-\frac{1}{r},r}(\pr\dom)}$, let $u_\epsilon$ be the solution to the equation
\begin{equation}\label{eqe}
\left\{\begin{array}{l}\nabla\cdot\left(\sigma\nabla u_\epsilon+a|\nabla u_\epsilon|^{p-2}\nabla u_\epsilon\right)=0,\\
u_\epsilon|_{\pr\dom}=\epsilon f.\end{array}\right.
\end{equation}
By Theorem \ref{thm-strong}, we have that
\begin{equation}
||{\epsilon}^{-1}u_\epsilon||_{W^{2,r}(\dom)}\leq C||f||_{W^{2-\frac{1}{r},r}(\pr\dom)}.
\end{equation}

We make the following Ansatz 
\begin{equation}\label{ansatz}
u_\epsilon=\epsilon u_0+{\epsilon}^{p-1}v_\epsilon,
\end{equation}
where $u_0\in W^{2,r}(\dom)$  satisfies the equation
\begin{equation}
\left\{\begin{array}{l} \nabla\cdot(\sigma\nabla u_0)=0,\\ u_0|_{\pr\dom}=f. \end{array}\right.
\end{equation}
 We have that
\begin{equation}
||u_0||_{W^{2,r}(\dom)}\leq C||f||_{W^{2-\frac{1}{r},r}(\pr\dom)}  .
\end{equation}

Note that $v_\epsilon$ satisfies
\begin{equation}
\left\{\begin{array}{l} \nabla\cdot(\sigma \nabla v_\epsilon) +{\epsilon}^{1-p}\nabla\cdot(a|\nabla u_\epsilon|^{p-2}\nabla u_\epsilon)=0,\\ v_\epsilon|_{\pr\dom}=0.\end{array}\right.
\end{equation}
By elliptic regularity
\begin{equation}
||v_\epsilon||_{W^{2,r}_0(\dom)}\leq C ||{\epsilon}^{-1}u_\epsilon||_{W^{2,r}(\dom)}^{p-1}\leq C||f||_{W^{2-\frac{1}{r},r}(\pr\dom)}^{p-1}.
\end{equation}
This shows that 
\begin{lem}
$v_\epsilon$ is bounded in $W^{2,r}_0(\dom)$ uniformly as ${\epsilon}\to0$. 
\end{lem}

\subsection{Asymptotics of the DN map}

The expansion of the solution $u_\epsilon$ gives us an expansion of the DN map in powers of $\epsilon$. Let $w\in W^{2,r}(\dom)$. 
Using Lemma \ref{vector-estimates} we obtain that
\begin{multline}
\left|\int_{\dom} a\left(|\nabla u_\epsilon|^{p-2}\nabla u_\epsilon-\epsilon^{p-1}|\nabla u_0|^{p-2}\nabla u_0    \right)\cdot \nabla \overline{w}   \right|\\[5pt]
\leq C \epsilon^{2p-3}\int_\dom \left( |\epsilon^{-1}\nabla u_\epsilon|^{p-2}+|\nabla u_0|^{p-2}  \right)|\nabla v_\epsilon|\,|\nabla w|=\mathscr{O}(\epsilon^{2p-3})
\end{multline}
Then
\begin{multline}
\la\Lambda_{\sigma,a,p}(\epsilon f),w|_{\pr\dom}\ra
= \int_\dom (\sigma+a|\nabla u_\epsilon|^{p-2})\nabla u_\epsilon\cdot\nabla \overline{w}
=\epsilon\int_\dom\sigma\nabla u_0\cdot\nabla  \overline{w}\\[5pt]
+\epsilon^{p-1}\int_\dom \sigma\nabla v_\epsilon\cdot\nabla  \overline{w}
+\epsilon^{p-1}\int_\dom a |\nabla u_0|^{p-2}\nabla u_0\cdot\nabla  \overline{w}+\mathscr{O}(\epsilon^{2p-3}).
\end{multline}
Here we can observe that, since $p>2$, we have
\begin{equation}
\lim_{\epsilon\to0}\frac{1}{\epsilon}\la\Lambda_{\sigma,a,p}(\epsilon f),w|_{\pr\dom}\ra=\int_\dom\sigma\nabla u_0\cdot\nabla  \overline{w}=\la\Lambda_\sigma(f),w\ra.
\end{equation}
The left hand side is exactly the DN map associated to the equation
\begin{equation}
\nabla\cdot(\sigma\nabla u_0)=0.
\end{equation}
We can observe here that, by the the result of \cite{N},  $\sigma$ can be recovered from $\Lambda_\sigma$, and hence from $\Lambda_{\sigma,a,p}$.

Suppose now that $w$ is such that $\nabla\cdot(\sigma\nabla w)=0$. Since $v_\epsilon\in W^{1,q}_0(\dom)$, 
\begin{equation}
\int_\dom \sigma\nabla v_\epsilon\cdot\nabla  \overline{w}=0.
\end{equation}
Define
\begin{multline}
I(u_0,{w})=\lim_{\epsilon\to0+}\epsilon^{1-p}\left[\la\Lambda_{\sigma,a,p}(\epsilon u_0|_{\pr\dom}),  \overline{w}|_{\pr\dom}\ra - \epsilon\int_\dom\sigma\nabla u_0\cdot\nabla  {w} \right]\\[5pt]
=\int_\dom a |\nabla u_0|^{p-2}\nabla u_0\cdot\nabla  {w}.
\end{multline}
With this notation we can write
\begin{equation}\label{dn-asymptotics}
\la\Lambda_{\sigma,a,p}(\epsilon f),w|_{\pr\dom}\ra
=\epsilon\la\Lambda_\sigma,w|_{\pr\dom}\ra+\epsilon^{p-1}I(u_0,\overline{w})+\mathscr{O}(\epsilon^{2p-3}).
\end{equation}

\subsection{Recovery of the parameter $a$}

Let the  functions $u_0, u_1, u_2, u_3\in C^{\infty}(\overline{\dom})$ be such that $\nabla\cdot(\sigma\nabla u_j)=0$, $j=0,1,2,3,4$.  We will further require that $\nabla u_1$ does not have any zeros.

If the DN map $\Lambda_{\sigma,a,p}$ is known, then $I(u_0,u_3)$  will also be known. 
 Note that
\begin{equation}
\left.\pr_{\bar z} \left[|\nabla(u_1+z \bar u_2)|^{p-2}\nabla(u_1+z \bar u_2)\right]\right|_{z=0}
=\frac{p-2}{2}|\nabla u_1|^{p-4}(\nabla u_1\cdot\nabla u_2)\nabla u_1,
\end{equation}
and
\begin{multline}
\left.\pr_{z} \left[|\nabla(u_1+z  u_2)|^{p-2}\nabla(u_1+z  u_2)\right]\right|_{z=0}\\[5pt]
=\frac{p-2}{2}|\nabla u_1|^{p-4}(\nabla\bar u_1\cdot\nabla u_2)\nabla u_1
+|\nabla u_1|^{p-2}\nabla  u_2
\end{multline}
Define
\begin{multline}
I(u_1,u_2,u_3)=\frac{2}{p-2}\pr_{\bar z}\left.I(u_1+z\bar u_2,u_3)\right|_{z=0}\\[5pt]
=\int_\dom a|\nabla u_1|^{p-4}(\nabla u_1\cdot\nabla u_2)(\nabla u_1\cdot\nabla u_3),
\end{multline}
and
\begin{multline}
J(u_1,u_2,u_3)=\frac{2}{p-2}\pr_{ z}\left.I(u_1+z u_2,u_3)\right|_{z=0}\\[5pt]
=\int_\dom a|\nabla u_1|^{p-4}\left[(\nabla \bar u_1\cdot\nabla u_2)(\nabla u_1\cdot\nabla u_3)
+\frac{p-2}{2}|\nabla u_1|^2\nabla u_2\cdot\nabla u_3\right], 
\end{multline}
This functionals are also known from the boundary data.

We will choose the solution $u_1$ to be \emph{real} valued and set $\beta_{u_1}=a |\nabla u_1|^{p-2}$. In this case we observe that
\begin{equation}\label{final-integral-id}
K(u_1,u_2,u_3)=\frac{2}{p-2}\left[J(u_1,u_2,u_3)-I(u_1,u_2,u_3)\right]=\int_\dom \beta_{u_1} \nabla u_2\cdot\nabla u_3
\end{equation}
is known from the boundary data.

We will recover $\beta_{u_1}$, for any $u_1$, from $K$. To do this, we will choose $u_2$ and $u_3$ to be complex geometric optics (CGO) solutions for the linear part of the equation. Such solutions were constructed in \cite{SU}. We recall their properties here.
\begin{prop}[see {\cite[Theorem 1.1]{SU}}]
If $\zeta\in\C^n$ is such that $\zeta\cdot\zeta=0$, then the equation
\begin{equation}
\nabla\cdot(\sigma\nabla u)=0
\end{equation}
has  complex geometric optics (CGO) solutions of the form
\begin{equation}
u(x)=e^{\zeta\cdot x}\sigma^{-1/2}(x)(1+r(x,\zeta)),
\end{equation}
where $r$ satisfies the equation
\begin{equation}\label{eq-r}
\triangle r+\zeta\cdot\nabla r-V r=V,\quad V=\frac{\triangle \sigma^{1/2}}{\sigma^{1/2}},
\end{equation}
and the estimates
\begin{equation}
||r||_{W^{s,2}(\dom)}\leq \frac{C_s}{|\zeta|}, \quad s>\frac{n}{2}.
\end{equation}
\end{prop}

 Now let  $\xi\in\R^{n}$ be arbitrary and choose $\eta,\mu\in S^{n-1}$ such that 
	\begin{equation}
	\xi\cdot\eta=\xi\cdot\mu=\eta\cdot\mu=0.
	\end{equation}
	Using these, we define $\zeta_{2},\zeta_{3}\in \C^{n}$ to be
	\begin{equation}\label{definition of zeta i}
	\zeta_{2}= \tau \mu-i\left(\frac{\xi}{2}+s\eta\right),\quad  
	\zeta_{3}=- \tau\mu-i\left(\frac{\xi}{2}-s\eta\right)
	\end{equation}
	where $s\in\R$ and $r$ satisfies 
	\begin{equation}
	\tau^{2}=\frac{\lvert\xi\rvert^{2}}{4}+s^{2}.
	\end{equation}
	With these choices, we have
	\begin{equation}
	\zeta_{2}\cdot\zeta_{2}=\zeta_{3}\cdot\zeta_{3}=0,\quad \zeta_2+\zeta_3=-i\xi.
	\end{equation}
We will take $u_2=e^{\zeta_2\cdot x}\sigma^{-1/2}(1+r_2)$, $u_3=e^{\zeta_3\cdot x}\sigma^{-1/2}(1+r_3)$ to be the corresponding CGO solutions.

Note that
\begin{equation}
\nabla u_i= e^{\zeta_i\cdot x}\left[ \zeta_i\sigma^{-\frac{1}{2}}(1+r_i)+\nabla (\sigma^{-\frac{1}{2}})(1+r_i)+\sigma^{-\frac{1}{2}}\nabla r_i   \right],
\end{equation}
so
\begin{multline}
\nabla u_2\cdot\nabla u_3 =e^{-i\xi\cdot x}\left[ \sigma^{-1}\zeta_2\cdot\zeta_3+\sigma^{-\frac{1}{2}}(\zeta_2+\zeta_3)\cdot\nabla(\sigma^{-\frac{1}{2}})(1+r_2)(1+r_3)\right.\\[5pt]
\left.
+|\nabla(\sigma^{-\frac{1}{2}})|^2+\sigma^{-1}(\zeta_2\cdot\nabla r_3+\zeta_3\cdot\nabla r_2)   \right]+\mathscr{O}(s^{-1})\\[5pt]
=e^{-i\xi\cdot x}[  \left( -\frac{1}{2}|\xi|^2 \right)\sigma^{-1} - i\sigma^{-\frac{1}{2}}\xi\cdot\nabla(\sigma^{-\frac{1}{2}}) +
|\nabla(\sigma^{-\frac{1}{2}})|^2\\[5pt]
+\sigma^{-1}(\zeta_2\cdot\nabla r_3+\zeta_3\cdot\nabla r_2)  ]+\mathscr{O}(s^{-1}).
\end{multline}
Here the $\mathscr{O}(s^{-1})$ is to be taken in the sense of $L^\infty(\dom)$ norms.

Using \eqref{eq-r}, we compute
\begin{equation}
\zeta_2\cdot\nabla r_3=(-i\xi-\zeta_3)\cdot\nabla r_3=-\xi\cdot\nabla r_3-V+\triangle r_3-V r_3=-V+\mathscr{O}(s^{-1}).
\end{equation}
Similarly
\begin{equation}
\zeta_3\cdot\nabla r_2=-V+\mathscr{O}(s^{-1}).
\end{equation}
Then 
\begin{multline}\label{u2u3-expansion}
\nabla u_2\cdot\nabla u_3 = 
e^{-i\xi\cdot x} [  \left( -\frac{1}{2}|\xi|^2 \right)\sigma^{-1} + i\sigma^{-\frac{3}{2}}\xi\cdot\nabla(\sigma^{\frac{1}{2}}) \\[5pt]+
\sigma^{-2}|\nabla(\sigma^{\frac{1}{2}})|^2-2V \sigma^{-1}  ]+\mathcal{O}(s^{-1}).
\end{multline}

We can now take the limit $s\to\infty$ in \eqref{final-integral-id} to obtain
\begin{multline}\label{s-limit}
\int e^{-i\xi\cdot x}\beta_{u_1}  [  \left( -\frac{1}{2}|\xi|^2 \right)\sigma^{-1} + i\sigma^{-\frac{3}{2}}\xi\cdot\nabla(\sigma^{\frac{1}{2}}) \\[5pt]+
\sigma^{-2}|\nabla(\sigma^{\frac{1}{2}})|^2-2V \sigma^{-1}  ]=\lim_{s\to\infty} K(u_1,u_2,u_3).
\end{multline}
Then
\begin{multline}
\frac{1}{2}\triangle(\sigma^{-\frac{1}{2}}\beta_{u_1})+\nabla\cdot\left( \sigma^{-1}\nabla(\sigma^{\frac{1}{2}})\beta_{u_1}\right)\\[5pt]+
\left( \sigma^{-\frac{3}{2}}|\nabla(\sigma^{\frac{1}{2}})|^2-2V \sigma^{-\frac{1}{2}}\right)\beta_{u_1} = \mathscr{F}^{-1}\lim_{s\to\infty} K(u_1,u_2,u_3),
\end{multline}
on $\mathbb{R}^n$, in the sense of distributions. In the identity above we have extended $\sigma$ so that it is $C^\infty$ on the whole of $\mathbb{R}^n$ and the support of $\sigma-1$ is compact. Since
\begin{multline}
\triangle(\sigma^{-\frac{1}{2}}\beta_{u_1})=\sigma^{-\frac{1}{2}}\triangle \beta_{u_1}-2\sigma^{-1}\nabla(\sigma^{\frac{1}{2}})\cdot\nabla\beta_{u_1}\\[5pt] +\left(2\sigma^{-\frac{3}{2}}|\nabla(\sigma^{\frac{1}{2}})|^2-\sigma^{-1}\triangle (\sigma^{\frac{1}{2}})  \right)\beta_{u_1}
\end{multline}
and
\begin{multline}
\nabla\cdot\left( \sigma^{-1}\nabla(\sigma^{\frac{1}{2}})\beta_{u_1} \right)=
 \sigma^{-1}\nabla(\sigma^{\frac{1}{2}})\cdot\nabla\beta_{u_1}+\left(\sigma^{-1}\triangle(\sigma^{\frac{1}{2}})-2\sigma^{-\frac{3}{2}}|\nabla\sigma^{\frac{1}{2}})|^2  \right)\beta_{u_1},
\end{multline}
we can conclude that $\beta_{u_1}$ satisfies
\begin{equation}\label{eq-beta}
\triangle \beta_{u_1}-3V\beta_{u_1}=2\sigma^{1/2} \mathscr{F}^{-1}\lim_{s\to\infty} K(u_1,u_2,u_3),
\end{equation}
on $\mathbb{R}^n$, in the sense of distributions. 
Suppose there exists $f\in \mathscr{E}'(\mathbb{R}^n)$ 
such that
\begin{equation}
\triangle f-3Vf=0.
\end{equation}
Then for any $\varphi\in C^\infty(\mathbb{R}^n)$ we have
\begin{equation}
\la f,\triangle\varphi-3V\varphi\ra=0.
\end{equation}
Since for any $\psi\in\mathscr{S}(\mathbb{R}^n)$ we can find $\varphi$ such that
\begin{equation}
\triangle\varphi-3V\varphi=\psi,
\end{equation}
it follows that $\la f,\psi\ra=0$ for all $\psi\in\mathscr{S}(\mathbb{R}^n)$, so $f=0$. 
This implies \eqref{eq-beta} can  have at most one compactly supported solution in distributions.
We can then invert the operator $\triangle-3V$ and we obtain the reconstruction formula
\begin{equation}
a=2|\nabla u_1|^{2-p}(\triangle-3V)^{-1}\left[ \sigma^{1/2} \mathscr{F}^{-1}\lim_{s\to\infty} K(u_1,u_2,u_3) \right].
\end{equation}

\section{Reconstruction when $p\in(1,2)$}\label{reconstruction-2}

In this section we give the reconstruction of the parameters $\sigma$, $a$, and $p$, in the case when $p\in(1,2)$. In this case the nonlinear term is sublinear, so we consider Dirichlet data of the form $\epsilon^{-1}f$, where $\epsilon\to0$. Once an asymptotic expansion for the DN map is established, the reconstruction method is the same as in the previous case.

\subsection{Asymptotics of solutions}

For  $0<\epsilon<1$, and $f\in \X_p=W^{\frac{1}{2},2}(\pr\dom)$, let $u_\epsilon$ be the solution to the equation
\begin{equation}\label{eqe2}
\left\{\begin{array}{l}\nabla\cdot\left(\sigma\nabla u_\epsilon+a|\nabla u_\epsilon|^{p-2}\nabla u_\epsilon\right)=0,\\
u_\epsilon|_{\pr\dom}=\epsilon^{-1} f.\end{array}\right.
\end{equation}
By Proposition \ref{weak-existence}, we have that
\begin{equation}
||\epsilon u_\epsilon||_{W^{1,2}(\dom)}\leq C||f||_{\X_p} .
\end{equation}

We make the following Ansatz 
\begin{equation}\label{ansatz2}
u_\epsilon=\epsilon^{-1} u_0+{\epsilon}^{1-p}v_\epsilon,
\end{equation}
where $u_0\in W^{1,2}(\dom)$  satisfies the equation
\begin{equation}
\left\{\begin{array}{l} \nabla\cdot(\sigma\nabla u_0)=0,\\ u_0|_{\pr\dom}=f. \end{array}\right.
\end{equation}
 We have that
\begin{equation}
||u_0||_{W^{1,2}(\dom)}\leq C||f||_{\X_p}  .
\end{equation}

Note that $v_\epsilon$ satisfies
\begin{equation}
\left\{\begin{array}{l} \nabla\cdot(\sigma \nabla v_\epsilon) +{\epsilon}^{p-1}\nabla\cdot(a|\nabla u_\epsilon|^{p-2}\nabla u_\epsilon)=0\\ v_\epsilon|_{\pr\dom}=0.\end{array}\right.
\end{equation}
Since $p<2$, we have that $|\nabla u_\epsilon|^{p-2}\nabla u_\epsilon\in L^2(\dom)$ and
\begin{equation}
||\,|\nabla u_\epsilon|^{p-2}\nabla u_\epsilon||_{L^2(\dom)}\leq||\nabla u_\epsilon||_{L^2(\dom)}^{p-1}.
\end{equation}
Elliptic regularity then gives
\begin{equation}
||v_\epsilon||_{W^{1,2}(\dom)}\leq C ||{\epsilon}u_\epsilon||_{W^{1,2}(\dom)}^{p-1}.
\end{equation}
this shows that 
\begin{lem}
$v_\epsilon$ is bounded in $W^{1,2}(\dom)$ uniformly as ${\epsilon}\to0$. 
\end{lem}

\subsection{Asymptotics of the DN map}

We will  expand the DN map in powers of $\epsilon$. Let $w\in W^{1,2}(\dom)$. 
Using Lemma \ref{vector-estimates} we obtain that
\begin{multline}
\left|\int_{\dom} a\left(|\nabla u_\epsilon|^{p-2}\nabla u_\epsilon-\epsilon^{1-p}|\nabla u_0|^{p-2}\nabla u_0    \right)\cdot \nabla \overline{w}   \right|\\[5pt]
\leq C \epsilon^{3-2p}\int_\dom \left( |\epsilon\nabla u_\epsilon|^{p-2}+|\nabla u_0|^{p-2}  \right)|\nabla v_\epsilon|\,|\nabla w|=\mathscr{O}(\epsilon^{2p-3}).
\end{multline}
Then
\begin{multline}
\la\Lambda_{\sigma,a,p}(\epsilon^{-1} f),w\ra
= \int_\dom (\sigma+a|\nabla u_\epsilon|^{p-2})\nabla u_\epsilon\cdot\nabla \bar w
=\epsilon^{-1}\int_\dom\sigma\nabla u_0\cdot\nabla \bar w\\[5pt]
+\epsilon^{1-p}\int_\dom \sigma\nabla v_\epsilon\cdot\nabla \bar w
+\epsilon^{1-p}\int_\dom a |\nabla u_0|^{p-2}\nabla u_0\cdot\nabla \bar w+\mathscr{O}(\epsilon^{3-2p}).
\end{multline}
Since $p\in(1,2)$, we have
\begin{equation}
\lim_{\epsilon\to0}\epsilon\la\Lambda_{\sigma,a,p}(\epsilon^{-1} f),w\ra=\int_\dom\sigma\nabla u_0\cdot\nabla \bar w
=\la\Lambda_\sigma(f),w\ra.
\end{equation}
The left hand side is exactly the DN map associated to the equation
\begin{equation}
\nabla\cdot(\sigma\nabla u_0)=0.
\end{equation}
By the the result of \cite{N}, we have that $\sigma$ can be recovered from $\Lambda_{\sigma,a,p}$.

Suppose now that $w$ is such that $\nabla(\sigma\nabla w)=0$. Since $v_\epsilon\in W^{1,2}_0(\dom)$, 
\begin{equation}
\int_\dom \sigma\nabla v_\epsilon\cdot\nabla \bar w=0.
\end{equation}
This leaves us with the expansion
\begin{multline}\label{dn-asymptotics2}
\la\Lambda_{\sigma,a,p}(\epsilon f),w\ra
=\epsilon^{-1}\int_\dom\sigma\nabla u_0\cdot\nabla \bar w\\[5pt]
+\epsilon^{1-p}\int_\dom a |\nabla u_0|^{p-2}\nabla u_0\cdot\nabla \bar w+\mathscr{O}(\epsilon^{3-2p}).
\end{multline}

Define
\begin{multline}
I(u_0,{w})=\lim_{\epsilon\to0+}\epsilon^{p-1}\left[\la\Lambda_{\sigma,a,p}(\epsilon^{-1} u_0|_{\pr\dom}),  \overline{w}|_{\pr\dom}\ra - \epsilon^{-1}\int_\dom\sigma\nabla u_0\cdot\nabla  {w} \right]\\[5pt]
=\int_\dom a |\nabla u_0|^{p-2}\nabla u_0\cdot\nabla  {w}.
\end{multline}
With this notation we can write
\begin{equation}\label{dn-asymptotics-lower}
\la\Lambda_{\sigma,a,p}(\epsilon^{-1} f),w|_{\pr\dom}\ra
=\epsilon^{-1}\la\Lambda_\sigma,w|_{\pr\dom}\ra+\epsilon^{1-p}I(u_0,\overline{w})+\mathscr{O}(\epsilon^{3-2p}).
\end{equation}
The coefficient $a$ can then be recovered from the functional $I$ using the same method as in the $p>2$ case.

\appendix

\section{Some vector estimates}\label{appendix-A}

Here we gather a few useful inequalities for vectors in $\C^n$. Some of these are well known, and we give them without a proof. We do give a proof to the estimate in Lemma \ref{holder}, which we need in order to prove the existence of strong solutions.

\begin{lem}\label{vector-estimates}
Suppose $\xi,\zeta\in\C^n$, $p\in(1,\infty)$, then
\begin{equation}
\left|\,|\xi|^{p-2}\xi -|\zeta|^{p-2}\zeta  \right|\leq C (|\xi|+|\zeta|)^{p-2}|\xi-\zeta|,
\end{equation}
\begin{equation}
\re\left[ (|\xi|^{p-2}\xi -|\zeta|^{p-2}\zeta)\cdot(\overline{\xi}-\overline{\zeta})  \right]
\sim (|\xi|+|\zeta|)^{p-2}|\xi-\zeta|^2,
\end{equation}
\begin{equation}
\left|\,|\xi|^{p}-|\zeta|^p\right|\leq p\left(|\xi|^{p-1}+|\zeta|^{p-1}  \right)|\xi-\zeta|.
\end{equation}
\end{lem}
These are standard, well known inequalities. A list containing them as well as others is collected in \cite[Appendix A]{SZ}.

\begin{lem}\label{p-2}
Let $p>2$, $\xi,\zeta\in\C^n$, $R>0$ such that $|\xi|, |\zeta|<R$, then
\begin{equation}
\left|\,|\xi|^{p-2}-|\zeta|^{p-2}\,\right|\leq C R^{\mu_1}|\xi-\zeta|^{\min(p-2,1)},
\end{equation}
where $C>0$, $\mu\geq0$ are constants that only depend on $p$.
\end{lem}
\begin{proof}
If $p\geq3$, we can apply Lemma \ref{vector-estimates} to obtain
\begin{equation}
\left|\,|\xi|^{p-2}-|\zeta|^{p-2}\,\right|\leq 2(p-2)R^{p-3}|\xi-\zeta|.
\end{equation}
Suppose now that $p\in(2,3)$, and also, without loss of generality, that $|\xi|\geq|\zeta|$. Then
\begin{multline}
\left|\,|\xi|^{p-2}-|\zeta|^{p-2}\,\right|
=|\zeta+\xi-\zeta|^{p-2}-|\zeta|^{p-2}\\[5pt]
\leq |\zeta|^{p-2}+|\xi-\zeta|^{p-2}-|\zeta|^{p-2}
=|\xi-\zeta|^{p-2}.
\end{multline}
\end{proof}

Finally, the following technical lemma will also be of use.
\begin{lem}\label{holder}
Let $p>2$, $\xi,\zeta\in\C^n$, $R>0$ such that $|\xi|, |\zeta|<R$,  and let $A$ and $B$ be the matrices with coefficients
\begin{equation}
A_{jk}=|\xi|^{p-4}\xi_j\overline{\xi_k}-|\zeta|^{p-4}\zeta_j\overline{\zeta_k},
\end{equation}
\begin{equation}
B_{jk}=|\xi|^{p-4}\xi_j{\xi_k}-|\zeta|^{p-4}\zeta_j{\zeta_k}.
\end{equation}
For any $n\times n$ matrix $H$ such that $H^T=H$, we have that
\begin{equation}
\left|\tr (AH+B\overline{H})\right|\leq CR^{\mu_2}||H||\,|\xi-\zeta|^{\frac{1}{4}\min({p-2},1)},
\end{equation}
where $||H||=\max_{jk}|H_{jk}|$ and $C,\mu_2>0$  are constants that depend only on $p$.
\end{lem}
\begin{proof}
The case when $p\geq3$ is relatively easier, so we will not write it out in full. For now assume that $2<p<3$.

First we define a new matrices $\tilde A=A-\frac{2}{n}\tr(A)I_{n\times n}$, $\tilde B=B-\frac{2}{n}\tr(B)I_{n\times n}$, with coefficients
\begin{equation}
\tilde A_{jk}=|\xi|^{p-4}(\xi_j\overline{\xi_k}-\frac{2}{n}|\xi|^2\delta_{jk})-|\zeta|^{p-4}(\zeta_j\overline{\zeta_k}-\frac{2}{n}|\zeta|^2\delta_{jk}),
\end{equation}
\begin{equation}
\tilde B_{jk}=|\xi|^{p-4}(\xi_j{\xi_k}-\frac{2}{n}\xi^2\delta_{jk})-|\zeta|^{p-4}(\zeta_j{\zeta_k}-\frac{2}{n}\zeta^2\delta_{jk}).
\end{equation}
Note that
\begin{multline}
\frac{1}{n}\left|\tr\left(\tr(A)I_{n\times n}H\right)\right|\\[5pt]
=\left|\,|\xi|^{p-2}-|\zeta|^{p-2}\right|\left|\tr(H)\right|
\leq \left|\tr(H)\right|\,|\xi-\zeta|^{p-2},
\end{multline}
and also 
\begin{equation}
\frac{1}{n}\left|\tr\left(\tr(B)I_{n\times n}\overline{H}\right)\right|=\left|\,|\xi|^{p-4}\xi^2-|\zeta|^{p-4}\zeta^2   \right|\,\left|\tr(\overline{H})\right|.
\end{equation}
We observe that
\begin{multline}
\left|\,|\xi|^{p-4}\xi^2-|\zeta|^{p-4}\zeta^2   \right|
=\left|\left(|\xi|^{\frac{p-4}{2}}\xi+|\zeta|^{\frac{p-4}{2}}\zeta \right)\cdot\left( |\xi|^{\frac{p-4}{2}}\xi-|\zeta|^{\frac{p-4}{2}}\zeta   \right)   \right|\\[5pt]
\leq \left( |\xi|^{\frac{p-2}{2}}+|\zeta|^{\frac{p-2}{2}} \right)
\left| \left( |\xi|^{\frac{p-4}{2}}\xi-|\zeta|^{\frac{p-4}{2}}\zeta   \right)\cdot
\left( |\xi|^{\frac{p-4}{2}}\overline{\xi}-|\zeta|^{\frac{p-4}{2}}\overline{\zeta}   \right)  \right|^{\frac{1}{2}}\\[5pt]
\leq 2R^{\frac{p-2}{2}}\left|\,|\xi|^{p-2}+|\zeta|^{p-2}-2|\xi|^{\frac{p-2}{2}}|\zeta|^{\frac{p-2}{2}}
+|\xi|^{\frac{p-4}{2}}|\zeta|^{\frac{p-4}{2}}\left(2|\xi|\,|\zeta|-\xi\cdot\overline{\zeta}-\overline{\xi}\cdot\zeta  \right)  \right|^{\frac{1}{2}}\\[5pt]
=2R^{\frac{p-2}{2}}\left|\left(|\xi|^{\frac{p-2}{2}}-|\zeta|^{\frac{p-2}{2}}\right)^2
+|\xi|^{\frac{p-4}{2}}|\zeta|^{\frac{p-4}{2}}\left(2|\xi|\,|\zeta|-\xi\cdot\overline{\zeta}-\overline{\xi}\cdot\zeta  \right)  \right|^{\frac{1}{2}}.
\end{multline}
In order to complete the estimate, suppose $|\xi|\leq|\zeta|$ and look at the following separately
\begin{multline}
\left|\,|\xi|\,|\zeta|-\xi\cdot\overline{\zeta}\right|
=\left|\,|\xi|\,|\zeta|-|\xi|^2+\xi\cdot(\overline{\xi}-\overline{\zeta})    \right|
\leq 2|\xi|\,|\xi-\zeta|\\[5pt]
\leq 2 |\xi|\left(|\xi|+|\zeta|  \right)^{1-\theta}|\xi-\zeta|^{\theta}
\leq 2^{2-\theta}|\xi|\,|\zeta|^{1-\theta}|\xi-\zeta|^{\theta},
\end{multline} 
where $\theta\in[0,1]$ can be chosen arbitrarily. 
We choose $\theta=\frac{p-2}{2}$. In general, after repeating this argument in all the cases that arise, we have
\begin{equation}
\left|2|\xi|\,|\zeta|-\xi\cdot\overline{\zeta}-\overline{\xi}\cdot\zeta\right|
\leq 2^{3-\frac{p-2}{2}}\min(|\xi|,|\zeta|)\max(|\xi|,|\zeta|)^{\frac{4-p}{2}}|\xi-\zeta|^{\frac{p-2}{2}}.
\end{equation}
We then have that
\begin{equation}
\left|\,|\xi|^{p-4}\xi^2-|\zeta|^{p-4}\zeta^2   \right|
\leq C R^{\frac{p-2}{2}}\left( |\xi-\zeta|^{\frac{p-2}{2}}+  R^{\frac{p-2}{4}} |\xi-\zeta|^{\frac{p-2}{4}}\right),
\end{equation}
so
\begin{equation}
\frac{1}{n}\left|\tr\left(\tr(B)I_{n\times n}\overline{H}\right)\right|
=C R^{\frac{3(p-2)}{4}}\left|\tr({H})\right|\,|\xi-\zeta|^{\frac{p-2}{4}}  .
\end{equation}

We can now replace the matrices $A$ and $B$ by $\tilde A$ and $\tilde B$, respectively. By the Cauchy-Schwartz inequality, we have
\begin{multline}
\left|\tr(\tilde AH+\tilde B\overline{H})\right|
\leq \left|\tr(\tilde AH)\right|+\left|\tr(\tilde B\overline{H})\right|\\[5pt]
\leq\left[\left(\tr(\tilde A{\tilde A}^\ast) \right)^{\frac{1}{2}}+ \left(\tr(\tilde B{\tilde B}^\ast) \right)^{\frac{1}{2}}\right] \left(\tr(H{H}^\ast)\right)^{\frac{1}{2}}\\[5pt]
\leq 2^{\frac{1}{2}}\left[\tr(\tilde A{\tilde A}^\ast) + \tr(\tilde B{\tilde B}^\ast) \right]^{\frac{1}{2}} \left(\tr(H{H}^\ast)\right)^{\frac{1}{2}}.
\end{multline}
We can compute that
\begin{equation}
\tr(\tilde A{\tilde A}^\ast)=\left( |\xi|^{p-2}-|\zeta|^{p-2} \right)^2+|\xi|^{p-4}|\zeta|^{p-4}\left(2|\xi|^2|\zeta|^2-2|\xi\cdot\overline{\zeta}|^2  \right),
\end{equation}
and
\begin{equation}
\tr(\tilde B{\tilde B}^\ast)=\left( |\xi|^{p-2}-|\zeta|^{p-2} \right)^2+|\xi|^{p-4}|\zeta|^{p-4}\left(2|\xi|^2|\zeta|^2-(\xi\cdot\overline{\zeta})^2-(\overline{\xi}\cdot{\zeta})^2  \right),
\end{equation}
Then
\begin{multline}
\left|\tr(\tilde A{\tilde A}^\ast)+\tr(\tilde B{\tilde B}^\ast)\right|
\leq 2 |\xi-\zeta|^{2(p-2)}\\[5pt]
+|\xi|^{p-4}|\zeta|^{p-4}\left[4|\xi|^2|\zeta|^2-(\xi\cdot\overline{\zeta})^2-(\overline{\xi}\cdot{\zeta})^2  -2(\xi\cdot\overline{\zeta})(\overline{\xi}\cdot{\zeta})\right]\\[5pt]
=2 |\xi-\zeta|^{2(p-2)} 
+|\xi|^{p-4}|\zeta|^{p-4}\left[4|\xi|^2|\zeta|^2-\left(\xi\cdot\overline{\zeta}+ \overline{\xi}\cdot{\zeta} \right)^2   \right].
\end{multline}
Notice that
\begin{multline}
\left|4|\xi|^2|\zeta|^2-\left(\xi\cdot\overline{\zeta}+ \overline{\xi}\cdot{\zeta} \right)^2   \right|\\[5pt]
=\left| 2|\xi|\,|\zeta|+\xi\cdot\overline{\zeta}+ \overline{\xi}\cdot{\zeta}\right|\,
\left| 2|\xi|\,|\zeta|-\xi\cdot\overline{\zeta}- \overline{\xi}\cdot{\zeta}\right|\\[5pt]
\leq 2^{5-\frac{p-2}{2}}|\xi|\,|\zeta|\min(|\xi|,|\zeta|)\max(|\xi|,|\zeta|)^{\frac{4-p}{2}}|\xi-\zeta|^{\frac{p-2}{2}}.
\end{multline}
It follows that 
\begin{equation}
\left|\tr(\tilde A{\tilde A}^\ast)+\tr(\tilde B{\tilde B}^\ast)\right|
\leq CR^{\frac{3(p-2)}{2}}|\xi-\zeta|^{\frac{p-2}{2}}.
\end{equation}
The lemma is now proven in the $p\in(2,3)$ case.

The main differences in the proof when $p\geq 3$ are that now
\begin{equation}
\left||\xi|^{p-2}-|\zeta|^{p-2}\right|\leq (p-2)(|\xi|^{p-3}+|\zeta|^{p-3})|\xi-\zeta|,
\end{equation}
and that it is enough to choose $\theta=\frac{1}{2}$.
\end{proof}

\bibliography{p-quasilinear}
\bibliographystyle{plain}

\end{document}